\documentclass[a4paper,12pt, twoside, reqno]{amsart}
\usepackage{amsmath}

\numberwithin{equation}{section}

\newtheorem{theorem}{Theorem}

\newtheorem{lem}{Lemma}

\newtheorem{defn}{Definition}
\newtheorem{ex}{Example}
\newtheorem{rem}{Remark}

\begin{document}
\title{Coupled solutions for a bivariate weakly nonexpansive operator  by iterations}
\author{V. Berinde, A. R. Khan and M. P\u acurar}

\begin{abstract}
We prove weak and strong convergence theorems for a double Krasnoselskij type iterative method to approximate coupled solutions of  a bivariate nonexpansive   operator $F:C\times C\rightarrow C$, where $C$ is a nonempty closed and convex subset of a Hilbert space. The new convergence theorems generalize, extend, improve and complement very important old and recent results in coupled fixed point theory.
Some appropriate examples to illustrate  our new results and their generalization are also given.
\end{abstract}

\maketitle

\pagestyle{myheadings}
\markboth{V. Berinde, A. R. Khan and M. P\u acurar}{Coupled solutions for a bivariate weakly nonexpansive operator}

\section{Introduction and preliminaries}

Let $X$ be a nonempty set. A pair $ \left(x,y\right) \in X \times X $ is called a \textit{coupled  fixed  point} of the mapping $F:X \times X \rightarrow X$ if                                             it is a solution of the system $$ F\left(x,y\right) = x, \quad F\left(y,x\right) = y.$$

The study of coupled fixed points has been considered in 2006 by Bhaskar and Lakshmikantham \cite{Bha} (see also \cite{Op86}). A rich literature on the existence of coupled fixed points of mixed monotone, monotone and non-monotone mappings, has been developed  ever since publication of that paper (see \cite{Abbas}-\cite{Agha12}, \cite{Amini}, \cite{Aydi}, \cite{Ber11}-\cite{Ber12b}, \cite{Ber11a}, \cite{Ber12c}-\cite{Ber12d}, \cite{Ciric12}-\cite{Hussain}, \cite{Kar}-\cite{Khan13b}, \cite{Ola}, \cite{Sabet}-\cite{Xiao}).

The novelty of this paper is that it considered coupled fixed point problem in a partially ordered metric space for mixed monotone mapping $F:X \times X \rightarrow X$ in conjunction with a contraction type condition of the form: 
\begin{equation} \label{Bhas}
	d\left(F\left(x,y\right),F\left(u,v\right)\right) \leq \frac{k}{2}\left[d\left(x,u\right) + d\left(y,v\right)\right],\textnormal{ for each }x \geq u, y \leq  v,
\end{equation}  
where $k \in \left[0,1\right)$.                                                                     

In particular,  the authors established  three kinds of coupled fixed point results: 1) existence theorems (Theorems 2.1 and 2.2); 2) existence and uniqueness theorem (Theorem 2.4); and 3) theorems that ensure the equality of the coupled fixed point components (Theorems 2.5 and 2.6). 

\begin{theorem} \em (\cite{Bha}, Theorem 2.1 and Theorem 2.6) \em \label{th4-1}
Let  $\left(X,\leq\right)$ be a partially ordered set and suppose there is a metric $d$  on $X$ such that $\left(X,d\right)$  is a complete metric space. Let $F : X \times X \rightarrow X $ be a continuous mapping having the mixed monotone property on $X$. 

If $F$  satisfies \eqref{Bhas} and  there exist $x_{0}, y_{0} \in X$ such that  
$$                                                                                       
x_{0} \leq F\left(x_{0},y_{0}\right)\textrm{ and }y_{0} \geq F\left(y_{0},x_{0}\right),
$$
 then there exist $\overline{x}, \overline{y} \in X$ such that $$\overline{x} = F\left(\overline{x},\overline{y} \right)\textnormal{ and }\overline{y}  = F\left(\overline{y} ,\overline{x}\right).$$ 
 Suppose, additionally, that $x_0,y_0 \in X$ are comparable. Then for the coupled fixed point $(\overline{x},\overline{y})$, we have $\overline{x} = \overline{y}$. 
\end{theorem}

Contraction type conditions arise naturally in connection with Lipshitzian properties of mappings in the study of nonlinear functional differential and integral equations. Therefore,  coupled fixed point results for contractions have  important applications in nonlinear analysis and have been applied successfully for solving  various classes of nonlinear functional equations: 1) integral equations  and systems of integral equations (\cite{Agha12}, \cite{Algham}, \cite{Aydi}, \cite{Berzig}, \cite{Gu}, \cite{Hussain}, \cite{Shat}, \cite{Sint}); (periodic) two point boundary value problems (\cite{Ber12b}, \cite{Bha}, \cite{Ciric12}, \cite{Urs});  nonlinear Hammerstein integral equations (\cite{Sang}); nonlinear elliptic problems and delayed hematopoesis models (\cite{Wu}); systems of differential and integral equations (\cite{Xiao}); nonlinear matrix and nonlinear quadratic equations (\cite{Agha}, \cite{Berzig}), initial value problems for ODE (\cite{Amini}; \cite{Samet}) etc. 

We note that in all the above mentioned cases, the main conclusion is drawn using (\cite{Bha}, Theorem 2.6) which guarantees existence as well as equality of components of the coupled fixed point.

On the other hand, in almost all the papers dealing with study of coupled fixed points, no attention is paid to the  constructive features of such a result, i.e., there is neither explicit mention of the method by which one could approximate  that coupled fixed point, nor on the order of convergence and / or error estimates of the iteration processes involved.

Moreover, there exist (mixed) monotone mappings $F\left(x,y\right)$ (see Examples \ref{ex-2} and \ref{ex-3} below), which possess coupled fixed points, for which no coupled fixed point theorem existing in literature can be applied. This is mainly  because all those theorems (we refer here only to the ones given in \cite{Agha}-\cite{Aydi}, \cite{Ber12b}, \cite{Berzig}, \cite{Bha}, \cite{Ciric12}-\cite{Hussain}, \cite{Sabet}, \cite{Samet}-\cite{Xiao})  are based on a strict contractive type condition  \eqref{Bhas}. 

All the  above observations motivate for constructive study of coupled fixed points of a bivariate mapping $F:X \times X \rightarrow X$ satisfying a weaker contractive condition of nonexpansiveness type
and  providing a constructive method to approximate these coupled fixed points which we generally meet in applications, i.e., when we have equality of the coupled fixed point components. 

The only paper that considers  asymptotically nonexpansive bivariate mappings and  the existence of their coupled fixed points is due to Olaoluwa et al. \cite{Ola}. No other attempt has been made to tackle this important problem. We find here coupled solutions for a bivariate weakly nonexpansive operator on Hilbert spaces through an iterative method. Since asymptotically nonexpansive (and nonexpansive)  bivariate mappings are particular sub-classes of the weakly nonexpansive mappings considered in the present paper, therefore our results also generalize, improve and complement the corresponding results obtained in \cite{Ola}.

In order to illustrate the broader scope and novelty of our results, we present appropriate examples to delineate them from the existing coupled fixed point theorems in literature and  indicate their potential use in applications.

\section{Nonexpansive bivariate operators}

In this paper, we define the concept of nonexpansiveness for bivariate mappings as follows.

\begin{defn}\rm  \label{def3}
Let $X$ be a normed linear  space and $C$ be a
subset of $X$. A mapping $F:C\times C\rightarrow X$ is called \textit{weakly nonexpansive} if 
\begin{equation} \label{nonexp-1}
\left\|F(x,y)-F(u,v)\right\| \leq  a\left\|x-u\right\|+b \left\|y-v\right\|,
\end{equation}
for all $x,y,u,v\in C$, where $a,b\geq 0$ and $a+b\leq 1$.
\end{defn}

 A similar but stronger  concept has been introduced in \cite{Ola}. 
 
 \begin{defn}\rm (\cite{Ola}) \label{def2}
Let $X$ be a normed linear  space and $C$ be a
subset of $X$. A mapping $F:C\times C\rightarrow X$ is called \textit{nonexpansive} if 
\begin{equation} \label{nonexp}
\left\|F(x,y)-F(u,v)\right\| \leq \frac{1}{2} \left(\left\|x-u\right\|+\left\|y-v\right\|\right),
\end{equation}
for all $x,y,u,v\in C$.
\end{defn}

Note that our condition \eqref{nonexp-1} is more general than \eqref{nonexp}:  any nonexpansive mapping $F$ is  weakly nonexpansive but the converse is not true, in general, as shown below.

\begin{ex} \label{ex-2}
Let $X=\mathbb{R}$ (with the usual metric)   and $F:X^2\rightarrow X$ be defined by  
$$
F(x,y)=\frac{x-2y}{3},\,\forall x,y\in X.  
$$
Then $F$ satisfies condition  \eqref{nonexp-1} but does not satisfy condition \eqref{nonexp}. Moreover, $F$ possesses a unique coupled fixed point, $(0,0)$, but no  coupled fixed point theorem established in \cite{Agha}-\cite{Aydi}, \cite{Ber12b}, \cite{Berzig}, \cite{Bha}, \cite{Ciric12}-\cite{Hussain}, \cite{Sabet}, \cite{Samet}-\cite{Xiao}   (and in other related papers) can be applied to this function $F$.

First, let us note that \eqref{nonexp-1} holds with the constants $a=\dfrac{1}{3}$ and $b=\dfrac{2}{3}$. Suppose  $F$  satisfies \eqref{nonexp}. 

Then, taking $x=y, y\neq z$ in \eqref{nonexp}, we get $\dfrac{1}{2}\geq \dfrac{2}{3}$, 
a contradiction. This proves that, indeed, $F$ does not satisfy \eqref{nonexp}.

To prove the last part of our claim, let us consider the contraction condition in  \cite{Sabet} (the same is valid for the corresponding conditions in \cite{Agha}-\cite{Aydi}, \cite{Ber12b}, \cite{Berzig}, \cite{Bha}, \cite{Ciric12}-\cite{Hussain}, \cite{Samet}-\cite{Xiao}),
\begin{equation} \label{eq-2k}
d\left(F\left(x,y\right),F\left(u,v\right)\right) \leq k d\left(x,u\right) + l d\left(y,v\right).
\end{equation}  
where $k,l \in \left(0,1\right)$  with   $k+l<1$.

Assume now that $F$ satisfies \eqref{eq-2k}.  Then,  taking $x=u, y\neq z$ in \eqref{eq-2k}, we get $l\geq \dfrac{2}{3}$ and taking $x\neq u, y= v$ in \eqref{eq-2k}, we get  $k\geq \dfrac{1}{3}$. Now these calculations for $k$ and $l$ lead to
$$
1\leq k+l<1,
$$
a contradiction. This proves that, indeed, $F$ does not satisfy the strict contraction condition \eqref{eq-2k}. This is also true for  contractive conditions considered in \cite{Agha}-\cite{Aydi}, \cite{Ber12b}, \cite{Berzig}, \cite{Bha}, \cite{Ciric12}-\cite{Hussain}, \cite{Sabet}, \cite{Samet}-\cite{Xiao}. 

Observe that for $F$ in this example, the double sequence $\{(x_n,y_n)\}_{n\geq 0}$, defined by the Picard-type iteration
\begin{equation} \label{dublu}
x_{n+1}=F(x_n,y_n),\quad y_{n+1}=F(y_n,x_n),\,n\geq 0,
\end{equation}
with $x_0,y_0 \in X$, is convergent but its limit is not  the coupled fixed point of $F$ (except for the  case  $x_0=y_0$); a fact which follows immediately from the expressions of $x_n$ and $y_n$:
$$
x_{n}=\frac{1}{2}\left[x_0-y_0+\left(-\frac{1}{3}\right)^{n}(x_0+y_0)\right],\,n\geq 0,
$$
$$
y_{n}=\frac{1}{2}\left[y_0-x_0+\left(-\frac{1}{3}\right)^{n}(x_0+y_0)\right],\,n\geq 0,
$$
\end{ex}
 It is important to note that  Opoitsev \cite{Op86} was the first who studied coupled fixed points of bivariate mappings (see also \cite{Op75}, \cite{Op84})  where  a double Picard-type iteration sequence $\{(x_n,y_n)\}_{n\geq 0}$ of the form \eqref{dublu} was used.

In order to state our main results, we need some concepts and results, adapted from the case of mono-variate  operators to the case of bivariate operators.  

The  concept of demicompact operator has been introduced by Petry-shyn \cite{Pet66} (see also \cite{Brow67b} and \cite{Pet73}) for a mapping $T:C\rightarrow H$,  where $C$ is a subset of a Hilbert space  $H$. For the bivariate case it is adapted  as follows.
 
\begin{defn}\rm  \label{def1}
 A mapping $F:C\times C\rightarrow H$ is called \textit{demicompact} if it has the property that whenever $\{u_{n}\}$ and $\{v_{n}\}$ are bounded sequences in $C$
with the property that  $\{F(u_{n},v_{n})-u_{n}\}$  and $\{F(v_{n},u_{n})-v_{n}\}$ converge strongly, then there exists a subsequence $\{(u_{n_{k}}, v_{n_{k}})\}$ of $\{(u_{n},v_{n})\}$ such that $u_{n_{k}}\rightarrow u$  and $v_{n_{k}}\rightarrow v$ strongly.
\end{defn}

We  need the following version of the well known Browder-Gohde-Kirk fixed point theorem  (see, for example, Theorem 3.1 in \cite{Ber07}), stated here in the Hilbert space  setting.

\begin{theorem} \label{th3}
Let $C$ be a bounded, closed and convex subset of a Hilbert space $H$ and let $F:C\times C\rightarrow C$  be a (weakly) nonexpansive  operator. Then $F$ has at least one coupled fixed point in $C$.

\end{theorem}

\begin{proof}
Let $T:C\rightarrow C$ be given by $T(x)=F(x,x),x\in C$. By the (weakly) nonexpansiveness property of $F$, we obtain the nonexpansiveness of $T$ and hence by Browder-Gohde-Kirk fixed point theorem, it follows that $Fix\,(T)\neq \emptyset$.

\end{proof}

\begin{rem}
Theorem \ref{th3} shows that $F$ has at least one (coupled) fixed point of the form $(\overline{x}, \overline{x})\in C\times C$, but in general, for a bivariate mapping $F$  it is  also possible to have coupled fixed points $(\overline{x},\overline{y})$ with unequal components, i.e., such that $\overline{x}\neq \overline{y}$, as shown by the following example.
\end{rem}

\begin{ex} \label{ex-1}
Let $X=\mathbb{R}$  (with the usual metric), $C=[-4,4]$   and $F:C^2\rightarrow X$ be defined by  
$$
F(x,y)=4-x^2-2y,\,\forall x,y\in C.  
$$
Then $F$ is weakly Lipschitzian with constants $a=8$ and $b=2$ (in the sense of Definition \ref{def3}) and $F$ possesses two coupled fixed points $(-4,-4)$, $(1,1)$ with equal components and two coupled fixed points with unequal components, $(-1,2)$ and $(2,-1)$.
\end{ex}

We close this section by stating another auxiliary result that will be needed.
\begin{lem} {\em (\cite{Ber07}, Lemma 1.8)} \label{lem1}
Let $x,y,z$ be points in a Hilbert space and $\lambda \in [0,1]$. Then
\[
\left\| \lambda x+(1-\lambda )y-z\right\| ^{2}=\lambda \left\|
x-z\right\| ^{2}+(1-\lambda )\left\| y-z\,\right\| ^{2}-\lambda
(1-\lambda )\left\| x-y\right\| ^{2}.
\]
\end{lem}

\section{Main results}

The main result of this paper is the following strong convergence theorem for a double Krasnoselskij-type algorithm associated with bivariate weakly nonexpansive operators on Hilbert spaces.

\begin{theorem} \label{th4}
Let $C$ be a bounded, closed and convex subset of a Hilbert space $H$ and let $F:C\times C\rightarrow C$  be weakly nonexpansive and demicompact operator. Then the set   of coupled fixed points of $F$  is nonempty and  the double iterative algorithm $\{(x_{n},x_{n})\}_{n=0}^{\infty }$ given by $x_{0}$ in $C$ and
\begin{equation} \label{eq-10}
x_{n+1} = \lambda x_n+(1-\lambda) F(x_n, x_n ),  n \geq 0,
\end{equation}
where $\lambda\in (0,1)$, converges (strongly) to a coupled fixed point  of $F$. 
\end{theorem}

\begin{proof} 

By Theorem \ref{th3}, $F$ has at least one coupled fixed point with equal components, $(\overline{x},\overline{x})\in C\times C$. 

We first show that the sequence $\{x_{n}-F(x_{n},x_{n})\}_{\,n\,\in \,\mathbb{N}}$ converges strongly to zero.

Indeed, by using Lemma \ref{lem1},
$$
\left\|x_{n+1}-\overline{x}\right\|^2=\left\|\lambda x_n+(1-\lambda) F(x_n, x_n )-\overline{x}\right\|^2\leq \lambda^2 \cdot \left\|x_n-\overline{x}\right\|^2
$$
\begin{equation} \label{eq-11}
+(1-\lambda)^2 \cdot \left\|F(x_n,x_n)-\overline{x}\right\|^2+2\lambda(1-\lambda)\left\langle F(x_n,x_n)-\overline{x},x_n-\overline{x}\right\rangle.
\end{equation}
On the other hand, 
\begin{equation} \label{eq-12}
\left\|x_n-F(x_n,x_n)\right\|^2= \left\|x_n-\overline{x}\right\|^2+ \left\|F(x_n,x_n)-\overline{x}\right\|^2-\left\langle F(x_n,x_n)-\overline{x},x_n-\overline{x}\right\rangle.
\end{equation}

By \eqref{eq-11}, \eqref{eq-12}, and  weak nonexpansiveness of $F$ and the fact that $F(\overline{x},\overline{x})=\overline{x}$, it follows that for any real number $a$ we have
\begin{equation*}
\left\| \,x_{n+1}-\overline{x}\right\| ^{2}+a^{2}\left\| \,x_{n}-F(x_{n},x_n)\right\|
^{2}\leq [2a^{2}+\lambda ^{2}+(1-\lambda )^{2}]\cdot \left\|
\,\,x_{n}-\overline{x}\right\| ^{2}+
\end{equation*}
\begin{equation}
+2[\lambda (1-\lambda )-a^{2}]\cdot \left\langle F(x_{n},x_n)-\overline{x},x_{n}-\overline{x}\right\rangle.
\end{equation} \label{eq-100a}
If we choose now a nonzero $a$ such that $a^{2}\leq \lambda (1-\lambda
)$, then from the last inequality we obtain
\begin{equation*}
\left\| \,x_{n+1}-\overline{x}\right\| ^{2}+a^{2}\left\| \,x_{n}-F(x_{n},x_n)\right\|
^{2}\leq
\end{equation*}
\begin{equation} \label{eq-100}
\leq \left( 2a^{2}+\lambda ^{2}+(1-\lambda )^{2}+2\lambda
(1-\lambda )-2a^{2}\right) \left\| x_{n}-\overline{x}\right\| ^{2}=\left\|
x_{n}-\overline{x}\right\| ^{2}.
\end{equation}
(we used the Cauchy-Schwarz inequality, 
\begin{equation*}
\left. \left\langle F(x_{n},x_n)-\overline{x},x_{n}-\overline{x}\right\rangle\,\,\leq \,\,\left\| \,F(x_{n},x_n)-\overline{x}\right\|
\cdot \left\| \,x_{n}-\overline{x}\right\| \leq \left\| \,\,x_{n}-\overline{x}\right\|
^{2}\right).
\end{equation*}
So, by  \eqref{eq-100} we get
\begin{equation} \label{eq-14}
 a^2\left\|x_n-F(x_n,x_n)\right\|^2\leq \left\|x_{n}-p\right\|^2-\left\|x_{n+1}-p\right\|^2, \,n \geq 0.
\end{equation}
By \eqref{eq-100} we deduce that  $\{\left\| \,\,x_{n}-\overline{x}\right\|\}$ is a decreasing 
sequence of non negative real numbers, hence it is convergent. 
By the inequality \eqref{eq-100}, we also have
$$
0\leq \left\|x_n-F(x_n,x_n)\right\|^2\leq \frac{1}{a^2}\left(\left\|x_{n}-p\right\|^2-\left\|x_{n+1}-p\right\|^2\right), \,n \geq 0.
$$
from which, by letting $n\rightarrow \infty$, we obtain
\begin{equation} \label{eq-16}
\lim_{n\rightarrow \infty} \left\|x_n-F(x_n,x_n)\right\|=0.
\end{equation}
This shows that $x_n-F(x_n,x_n)\rightarrow 0$ (strongly) and so it follows by demicompactness of $F$ that there exists a subsequence  $\{x_{n_k}\}\subset C$ and a point $q\in C$ such that
$$
\lim_{k\rightarrow \infty} x_{n_k}=q.
$$
As $F$ is nonexpansive, so it is continuous. This implies
$$
\lim_{k\rightarrow \infty} F (x_{n_k},x_{n_k})=F(q,q).
$$
By \eqref{eq-16}, 
$
0=\lim_{k\rightarrow \infty} \left(x_{n_k}- F(x_{n_k},x_{n_k})\right)=q-F(q,q),
$
which shows that $(q,q)$ is a coupled fixed point of $F$.

Using now the inequality \eqref{eq-14}, with $\overline{x}=q$, we deduce that the sequence of nonnegative real numbers $\{\left\|x_n-q\right\|\}_{n\geq 0}$ is nonincreasing, hence convergent. 

Since its subsequence $\{\left\|x_{n_k}-q\right\|\}_{k\geq 0}$ converges to $0$, it follows that the  sequence $\{\left\|x_n-q\right\|\}_{n\geq 0}$ itself converges to $0$, that is, the sequence $\{(x_n,x_n)\}$ converges strongly to $(q,q)$, as $n\rightarrow \infty$.
\end{proof}

\begin{rem}
Any nonexpansive bivariate mapping is weakly nonexpansive. Hence, by Theorem \ref{th4}, we obtain Corollary 2.3 in \cite{Ola}.
\end{rem}

We now introduce the concept of demicompactness at a point for a bivariate operator (adapted from the original definition of Petryshyn \cite{Pet66}).

\begin{defn} \rm
 A map $F$ of $C\times C\subset H$ into $H$ is said to be \textit{demicompact at} $(u,u)$ if, for any bounded  sequence $\{x_n\}$ in $C$ such that $x_n-F( x_n,x_n)\rightarrow (u,u)$ as $n\rightarrow \infty$, there exists a subsequence $\{x_{n_j}\}$ and an $x$ in $C$ such that $x_{n_j}\rightarrow x$ as $j\rightarrow \infty$ and $x-F(x,x)=u.$ 
\end{defn}

\begin{rem} \rm
Clearly, if $F$ is demicompact on $C$, then it is demicompact at $0$ but the converse is not true. 

The demicompactness  of $F$ on the whole $C$ in Theorem \ref{th4} may be weakened to the demicompactness at $0$.

\end{rem}

\begin{theorem} \label{th5}
Let $H$ be a Hilbert space, $C$ a closed, bounded and convex subset of $H$, and $F:C\times C\rightarrow C$ a weakly nonexpansive mapping such that $F$ is demicompact at $0$.

Then the Krasnoselkij-type double sequence $\{(x_{n},x_n)\}_{n=0}^{\infty }$ given by $x_{0}$ in $C$ and \eqref{eq-10} converges (strongly) to a coupled fixed point of $F$.
\end{theorem}

\begin{proof}
Note that in the proof of Theorem \ref{th4}, we actually used the demicompactness of $F$ at $0$, so the arguments used there can be applied here.
\end{proof}

\begin{rem} \rm
The conclusion of Theorem \ref{th5} remains true if instead of the demicompactness of $F$ at $0$, we suppose $I-F(x,x)$ maps closed sets in $C$ into closed sets of $H$ (see \cite{Pet66}).
\end{rem}

If in Theorems \ref{th4} and \ref{th5},  we remove the  demicompactness assumption, then (see \cite{Ber07}), the  Krasnoselskij iteration does no longer converge strongly, in general,
but it could converge (at least) weakly to a fixed point, as shown in the next theorem, which extends  Theorem 3.3 in \cite{Ber07}.

Denote by $Fix\,(F)$,  the set of all coupled fixed points of $F$ with equal components, i.e., $Fix\,(F)=\{p\in C :\,F(p,p)=p\}$.

\begin{theorem} \label{th5-1}
Let $H$ be a Hilbert space, $C$ a closed, bounded and convex subset of $H$, and $F:C\times C \rightarrow C$ a weakly nonexpansive mapping such that  $Fix\,(F)=\{(p,p)\}$. Then the Krasnoselskij iteration $\{x_{n}\}_{n=0}^{\infty }$ given by $x_{0}$ in $C$ and
\begin{equation} \label{eq-21}
x_{n+1} = (1-\lambda) x_n+\lambda F(x_n ,x_n),  n \geq 0,
\end{equation}
converges weakly to $p$, for any $\lambda\in (0,1)$.
\end{theorem}

\begin{proof} 
It suffices to show that if $\{x_{n_{j}}\}_{j=0}^{\infty },\ x_{n_{j}}=T^{n_{j}} x$, where $T x=(1-\lambda) x+\lambda F(x_,x)$, converges weakly to a certain $
p_{0}$, then $p_{0}$ is a fixed point of $T$ (and hence of $F$) and therefore $p_{0}=p.$ Suppose that $\{x_{n_{j}}\}_{j=0}^{\infty
}$ does not converge weakly to $p$. As  $F$ is weakly nonexpansive, so we have 
\begin{equation*}
\left\|Tx\!-\!Ty\right\|\!\leq\! \lambda \left\|x\!-\!y\right\|+(1\!-\!\lambda) \left\|F(x,x)\!-\!F(y,y)\right\|\!
\end{equation*}
$$
\leq\! \lambda \left\|x\!-\!y\right\|+(1\!-\!\lambda) \left\|x\!-\!y\right\|\!=\!\left\|x\!-\!y\right\|,
$$
which shows that $T$ is nonexpansive and hence we get:
\begin{equation*}
\left\| \,\,x_{n_{j}}-T p_{0}\right\| \leq \left\| \,\,Tx_{n_{j}}-Tp_{0}\right\| +\left\| \,\,x_{n_{j}}-Tx_{n_{j}}\right\| \leq
\end{equation*}
\begin{equation*}
\leq \left\| \,\,x_{n_{j}}-p_{0}\right\| +\left\|
\,\,x_{n_{j}}-Tx_{n_{j}}\right\|.
\end{equation*}
Using the arguments in the proof of Theorem \ref{th3}, it follows
\begin{equation*}
\left\| \,\,x_{n_{j}}-Tx_{n_{j}}\right\| \rightarrow 0,\,\,\,\,\,%
\text{as}\,\,\,\,\,\,n\rightarrow \infty ,
\end{equation*}
and so the last inequality implies that
\begin{equation} \label{eq-24}
\lim \,\,\sup \left( \,\left\| \,\,x_{n_{j}}-Tp_{0}\right\|
-\,\left\| \,\,x_{n_{j}}-p_{0}\right\| \,\right) \leq 0.  
\end{equation}
As in the proof of Theorem \ref{th3}, we have
\begin{equation*}
\left\| \,\,x_{n_{j}}-Tp_{0}\right\| ^{2}=\left\|
\,\,(x_{n_{j}}-p_{0})+(p_{0}-Tp_{0})\right\| ^{2}=
\end{equation*}
\begin{equation*}
=\left\| \,\,x_{n_{j}}-p_{0}\right\| ^{2}+\left\| p_{0}-Tp_{0}\right\| ^{2}+2\left\langle x_{n_{j}}-p_{0},p_{0}-Tp_{0}\right\rangle,
\end{equation*}
which shows, together with $x_{n_{j}}\rightharpoonup p_{0}$ (as $%
j\rightarrow \infty $), that
\begin{equation} \label{eq-22}
\underset{n\rightarrow \infty }{\lim }\left[ \left\|
\,\,x_{n_{j}}-Tp_{0}\right\| ^{2}-\left\| x_{n_{j}}-p_{0}\right\|
^{2}\,\right] =\left\| p_{0}-Tp_{0}\right\| ^{2}.  
\end{equation}
On the other hand, we have
\begin{equation}\label{eq-23}
\left\| \,\,x_{n_{j}}-Tp_{0}\right\| ^{2}-\left\|
x_{n_{j}}-p_{0}\right\| ^{2}=\left( \,\left\| \,\,x_{n_{j}}-Tp_{0}\right\| -\left\| \,\,x_{n_{j}}-p_{0}\right\| \,\right) \cdot
\end{equation}
\begin{equation*} 
\cdot \left( \,\left\| \,\,x_{n_{j}}-Tp_{0}\right\| +\left\|
\,\,x_{n_{j}}-p_{0}\right\| \,\right) .  
\end{equation*}
Since $C$ is bounded, the sequence $$\left\{ \left\|
\,\,x_{n_{j}}-Tp_{0}\right\| +\,\left\|
\,\,x_{n_{j}}-p_{0}\right\| \right\} \,\,$$ is bounded, too, and so by
the relations \eqref{eq-24} -- \eqref{eq-23} we get
\begin{equation*}
\left\| p_{0}-Tp_{0}\right\| \leq 0,\,\,\,\,\,\,\,\text{%
i.e.,\thinspace }\,\,\,\,\,\,\,\,\,\,Tp_{0}=p_{0}\Leftrightarrow p_{0}=F(p_0,p_0)=p.
\end{equation*}
\end{proof}

 \begin{rem} \rm The assumption $Fix\,(F)=\{(p,p)\}$ in Theorem \ref{th5-1} may be
removed  to obtain the following more general result (similar to Theorem 3.4 in \cite{Ber07}).
 \end{rem}

\begin{theorem} \label{th6}

Let $C$ be a bounded, closed and convex subset of a Hilbert space and $F:C\times C\rightarrow
C$ be weakly nonexpansive operator. Then the Krasnoselskij algorithm $\{x_{n}\}_{n=0}^{\infty }$ given by $x_{0}$ in $C$ and
\begin{equation*} 
x_{n+1} = (1-\lambda) x_n+\lambda F(x_n ,x_n),  n \geq 0,
\end{equation*}
converges weakly to a coupled fixed point of $F$.
\end{theorem}

\begin{proof} We  essentially follow the steps and arguments of the proof of Theorem 3.4 in \cite{Ber07}. For each $p\in Fix\,(F)$ and each $n$, we have, as in the proof of Theorem \ref{th3},
\begin{equation*}
\left\| \,\,x_{n+1}-p\right\| \leq \left\| \,\,x_{n}-p\right\| ,
\end{equation*}
which shows that the function $g(p)=\underset{n\rightarrow \infty }{\lim }\left\| \,\,x_{n}-p\right\| $ is well defined and is a lower semicontinuous convex function on the nonempty convex set $Fix\,(F)$. Let
\begin{equation*}
d_{0}=\inf \{g(p):p\in Fix\,(F)\}.
\end{equation*}
For each $\varepsilon >0$, the set
\begin{equation*}
E_{\varepsilon }=\{y\,:g(y)\leq d_{0}+\varepsilon \}
\end{equation*}
is closed, convex, and, hence, weakly compact. 

Therefore $\underset{\varepsilon >0}{\cap }E_{\varepsilon
}\neq \emptyset $ (in fact
$
\underset{\varepsilon >0}{\cap }E_{\varepsilon }=\{y\,:\,g(y)=d_{0}\}\equiv
E_{0}).
$
Moreover, $F_{0}$ contains exactly one point. Indeed, since $F_{0}$ is
convex and closed, for \mbox{$p_{0},p_{1}\!\in\! F_{0}$}, and $p_{\lambda }=(1-\lambda
)p_{0}+\lambda p_{1},$%
\begin{equation*}
g^{2}(p_{\lambda })=\underset{n\rightarrow \infty }{\lim }\left\|
\,\,p_{\lambda }-x_{n}\right\| ^{2}=\underset{n\rightarrow \infty }{\lim }%
(\left\| \lambda (p_{1}-x_{n})+(1-\lambda )(p_{0}-x_{n})\right\| ^{2})=
\end{equation*}
\begin{equation*}
=\underset{n\rightarrow \infty }{\lim }(\lambda ^{2}\left\|
p_{1}-x_{n}\right\| ^{2}+(1-\lambda )^{2}\left\| \,p_{0}-x_{n}\right\| ^{2}+2\lambda (1-\lambda )\left\langle p_{1}-x_{n},p_{0}-x_{n}\right\rangle)
\end{equation*}
\begin{equation*}
=\underset{n\rightarrow
\infty }{\lim }(\lambda ^{2}\left\| p_{1}-x_{n}\right\| ^{2}+(1-\lambda )^{2}\left\| \,p_{0}-x_{n}\right\| ^{2}+2\lambda (1-\lambda
)\left\| \,p_{1}-x_{n}\right\| \cdot \left\| \,p_{0}-x_{n}\right\|)+ 
\end{equation*}
\begin{equation*}
+\underset{n\rightarrow \infty }{\lim }\left\{ 2\lambda (1-\lambda
)\,\,[\left\langle p_{1}-x_{n},p_{0}-x_{n}\right\rangle-\left\| \,\,p_{1}-x_{n}\right\| \cdot
\left\| \,\,p_{0}-x_{n}\right\| ]\right\} =
\end{equation*}
\begin{equation*}
=g^{2}(p)+\underset{n\rightarrow \infty }{\lim }\left\{ 2\lambda (1-\lambda
)\,\,\left\langle p_{1}-x_{n},p_{0}-x_{n}\right\rangle-\left\| \,\,p_{1}-x_{n}\right\| \cdot \left\|
\,\,p_{0}-x_{n}\right\| \right\} .
\end{equation*}
Hence
\begin{equation*}
\underset{n\rightarrow \infty }{\lim }\left\{ 2\lambda (1-\lambda
)\,\,[\left\langle p_{1}-x_{n},p_{0}-x_{n}\right\rangle-\left\| \,\,p_{1}-x_{n}\right\| \cdot
\left\| \,\,p_{0}-x_{n}\right\| ]\right\} =0.
\end{equation*}
Since
\begin{equation*}
\left\| p_{1}-x_{n}\right\| \rightarrow d_{0}\text{ and }\left\|
\,\,p_{0}-x_{n}\right\| \rightarrow d_{0},
\end{equation*}
the later relation implies that
\begin{equation*}
\left\| \,\,p_{1}-p_{0}\right\| ^{2}=\left\|
\,\,(p_{1}-x_{n})+(x_{n}-p_{0}\right\| ^{2}=\left\|
\,\,p_{1}-x_{n}\right\| ^{2}+\,
\end{equation*}
\begin{equation*}
+\left\| \,\,x_{n}-p_{0}\right\|
^{2}-2\,<p_{1}-x_{n},\,p_{0}-x_{n}>\rightarrow
d_{0}^{2}+d_{0}^{2}-2d_{0}^{2}=0,
\end{equation*}
giving a contradiction.

Now, in order to show that $x_{n}=F^{n}(x_{0},x_0)\rightharpoonup p_{0},$ it suffices to assume that
$x_{n_{j}}\rightharpoonup p$ for an infinite
subsequence and then prove that $p=p_{0}$. By the arguments in the proof of Theorem \ref{th4}, $%
p\in Fix\, (F)$. Considering the definition of $g$ and the fact that
$x_{n_{j}}\rightarrow p$, we have
\begin{equation*}
\left\| \,\,x_{n_{j}}-p_{0}\right\| ^{2}=\left\|
\,\,x_{n_{j}}-p+p-p_{0}\right\| ^{2}=\left\| \,\,x_{n_j}-p\right\|
^{2}+\left\| \,\,p-p_{0}\right\| ^{2}-
\end{equation*}
\begin{equation*}
-2\,\left\langle x_{n_{j}}-p,\,p-p_{0}\right\rangle\rightarrow g^{2}(p)+\left\| \,\,p-p_{0}\right\|
^{2}=g^{2}(p_{0})=d_{0}^{2}.
\end{equation*}
Since $g^{2}(p)\geq d_{0}^{2},$ the last inequality implies that
\begin{equation*}
\left\| \,\,p-p_{0}\right\| \leq 0,
\end{equation*}
which means that $p=p_{0}.$

\end{proof}

\section{Conclusions and further study}

\begin{ex} \label{ex-3}
Let $X=\mathbb{R}$ (with the usual metric), $C=[-1,1]$. Define  bivariate function $F:C^2\rightarrow C$  by  
$$
F(x,y)=-\frac{x+y}{2},\,\forall x,y\in C.  
$$
Then $F$ satisfies \eqref{nonexp-1} and is demicompact. Hence, all the assumptions of Theorem \ref{th3} are satisfied. It is easy to see that $F$ possesses a unique coupled fixed point, $(0,0)$, and the Krasnoselskij-type iteration algorithm \eqref{eq-10} yields the  sequence
$$
x_n=\left(1-2\lambda\right)^n x_0,n\geq 0.
$$
Since $-1<1-2\lambda<1$, it follows that $(x_n,x_n)$ converges to $(0,0)$ as $n\rightarrow \infty$, for any initial value $x_0$.

This shows that, for weakly nonexpansive mappings,  by using a Krasnoselskij-type iteration we can reach the convergence,  while, by means of Picard-type iterations, this cannot be obtained, in general. Indeed,  in this case, the Picard-type iteration $(u_n,u_n)$ associated with $F$ is given by $u_{n+1}=-u_n,\,n\geq 0$, which is not convergent (except for the case $u_0=0$).

\end{ex}

\begin{rem}

It is important at this stage to say that  the coupled fixed point theorems existing in literature, see \cite{Agha}-\cite{Aydi}, \cite{Ber12b}, \cite{Berzig}, \cite{Bha}, \cite{Ciric12}-\cite{Hussain}, \cite{Sabet}, \cite{Samet}-\cite{Xiao} (only a short list cited here), cannot be applied to the bivariate functions in Examples \ref{ex-2} and \ref{ex-3}.

Finally, let us note that the double sequence $\{(x_n,y_n)\}$, defined for each component by a formula of the form  \eqref{eq-10} with $F(x_n,y_n)$ and $F(y_n,x_n)$, respectively, instead of $F(x_n,x_n)$, in the  case of the function $F$ in Example \ref{ex-2} will be given by
$$
x_{n}=\frac{1}{2}\left[(1-\lambda)^n(x_0-y_0)+(1-2\lambda)^{n}(x_0+y_0)\right],\,n\geq 0,
$$
$$
y_{n}=\frac{1}{2}\left[(1-\lambda)^n(y_0-x_0)+(1-2\lambda)^{n}(x_0+y_0)\right],\,n\geq 0,
$$
and it is easily seen that $\{(x_n,y_n)\}$ still converges to $(0,0)$,  the unique coupled fixed point of $F$, for all $x_0,y_0\in C$. 

This  also indicates that it is not necessary to consider only the case of a double sequence with equal components  $\{(x_n,x_n)\}$  in Theorems \ref{th4}-\ref{th6} (but the proof of a convergence theorem for such an iterative method will be essentially different from the one given in this paper).
\end{rem}

To conclude this paper, we note that, for the general case of a weakly nonexpansive bivariate mapping $F$, the Picard-type iteration process \eqref{dublu} does not generally converge or, even if it converges, its limit is not a coupled fixed point of $F$, but the Krasnoselskij type iteration process always converges to a coupled fixed point of $F$.

In the same way, we can prove convergence theorems for iterative methods of Krasnoselskij type for tripled fixed points, quadruple fixed points etc. of weakly nonexpansive mappings (see \cite{Ber11}-\cite{Ber12a}, \cite{Ber11a}, \cite{Ber12d},\cite{BerK}, \cite{Chid09}, \cite{Chid10}, \cite{Kar}-\cite{MarM}, \cite{Pac-teza}, \cite{Pac12}, \cite{Raf}-\cite{RusFPT}, and references therein). 

A similar approach for other contractive conditions and   algorithms existing in literature (see \cite{Ber07}, \cite{Chid09}) will be considered in our future work.

\section*{Acknowledgements}

The paper has been finalized during the first author's visit of Department of Mathematics and Statistics, King Fahd University of Petroleum and Minerals. He gratefully thanks the host for kind hospitality and excellent work facilities offered.
The first and third authors' research was supported by the Grants PN-II-RU-TE-2011-3-239 and PN-II-ID-PCE-2011-3-0087 of the Romanian Ministry of Education
and Research. 

The second author is grateful to KACST, Riyad, for supporting research project J-P-11-0623.

\vskip 0.5 cm {\it Department of Mathematics and Computer Science

North University Center at Baia Mare

Technical University of Cluj-Napoca

Victorie1 76, 430072 Baia Mare ROMANIA

E-mail: vasile\_berinde@yahoo.com }

\vskip 0.5 cm {\it Department of Mathematics ans Statistics

King Fahd University of Petroleum and Minerals

Dhahran, Saudi Arabia

E-mail: arahim@kfupm.edu.sa}

\vskip 0.5 cm {\it Department of Analysis, Forecast and Mathematics

Faculty of Economics and Bussiness Administration

Babe\c s-Bolyai University of Cluj-Napoca

56-60 T. Mihali St., 400591 Cluj-Napoca ROMANIA

E-mail: madalina\_pacurar@yahoo.com}

\end{document}